\documentclass[12pt, a4paper, leqno]{amsart}

\usepackage[utf8]{inputenc}
\usepackage{amsmath} 
\usepackage{amsfonts}
\usepackage{amssymb}
\usepackage{txfonts}   
\usepackage{amsfonts}     
\usepackage{graphicx}     
\usepackage[dvipsnames]{xcolor}
\usepackage[colorlinks, allcolors=RedViolet]{hyperref}

\newcommand{\R}{\varmathbb{R}}
\newcommand{\Z}{\varmathbb{Z}}

\newcommand{\Rn}{{\varmathbb{R}^n}}

\newcommand{\Ha}{\mathcal{H}}
\newcommand{\M}{\mathcal{M}}
\newcommand{\Mu}{\widetilde{\mathcal{M}}}

\newcommand{\Po}{\mathcal{P}}

\newcommand{\Ri}{\mathcal{R}}

\newcommand{\ve}{\varepsilon}

\def\dist{\qopname\relax o{dist}}

\def\inte{\qopname\relax o{int}}

\def\phi{\varphi}

\def\L{{\textsc{n}L}}

\let\oldmarginpar\marginpar
\renewcommand\marginpar[1]{\-\oldmarginpar[\raggedleft\footnotesize #1]%
{\raggedright\footnotesize #1}}

\usepackage{amsthm}

\swapnumbers
\theoremstyle{plain}
\newtheorem{theorem}[equation]{Theorem}
\newtheorem{lemma}[equation]{Lemma}
\newtheorem{proposition}[equation]{Proposition}
\newtheorem{corollary}[equation]{Corollary}

\theoremstyle{definition}

\theoremstyle{remark}
\newtheorem{remark}[equation]{Remark}

\numberwithin{equation}{section}

\title[On Hausdorff content maximal operator and Riesz potential]{On Hausdorff content maximal operator and Riesz potential for non-measurable functions}

\author{Petteri Harjulehto}
\address[Petteri Harjulehto]{Department of Mathematics and Statistics,
FI-00014 University of Helsinki, Finland}
\email{petteri.harjulehto@helsinki.fi}

\author{Ritva Hurri-Syrj\"anen}
\address[Ritva Hurri-Syrj\"anen]{Department of Mathematics and Statistics,
FI-00014 University of Helsinki, Finland}
\email{ritva.hurri-syrjanen@helsinki.fi}

\date{\today}

\begin{document}


\begin{abstract} 
We introduce Riesz potentials for non-Lebesgue measurable functions by taking  the integrals  in the sense of Choquet 
with respect to Hausdorff content and
prove boundedness results for these operators. 
Some earlier results are recovered or extended 
now using integrals taken in the sense of Choquet with respect to Hausdorff content.
Some earlier results also for maximal operators are considered, but now for non-measurable functions.
\end{abstract}

\maketitle

\section{Introduction}

The Riemann and Lebesgue integral theories  were developed for Lebesgue measurable functions.
 Ways to study
these integral theories  when functions are not Lebesgue measurable  appear in  \cite{Hil17, Jef32, Zak66}.
Also, the
Choquet integral theory  has been 
concentrated to the context when functions are  continuous  or quasicontinuous  or at least Lebesgue measurable. However, this 
 seems not to be necessary
 in many cases when considering properties of Choquet integrals with respect to Hausdorff content.
D.\ Denneberg  studies Choquet integrals  in his 
monograph  \cite{Den94}  extensively  without assuming that  functions are Lebesgue measurable. 
The properties of Choquet integrals with respect to capacities
are stated in  Denneberg's monograph \cite{Den94} with minimal assumptions on capacities 
and functions.
We are interested in Choquet integrals with respect to Hausdorff content, that is Hausdorff capacity.
By \cite[Theorem~6.3]{Den94}
Choquet integral with respect to Hausdorff content is quasi-sublinear  for  all non-negative functions. 
We refer to Theorem~\ref{DennebergThm}. 
There is no need to assume that functions are Lebesgue measurable. 

The Choquet integral  was introduced by G.\ Choquet \cite{Cho53} and  applied  by D.\ R.\ Adams  in the study of
nonlinear potential theory \cite{Adams1986, Adams1998, Adams2015, AdaX2020}. Later on, 
J. Xiao has studied Choquet integrals extensively with his coauthors \cite{DX, SXY}, as well as J. Kawabe \cite{Kawabe2019}.
H. Saito, H. Tanaka,  and T. Watanabe have  results on the classical maximal  operators 
when  integrals  are taken in sense of Choquet with respect to Hausdorff content
\cite{SaitoTanakaWatanabe2016, Saito2019, SaitoTanakaWatanabe2019, SaitoTanaka2022}.
Recently, there has arisen a new interest in Choquet integrals and their properties. We refer to
\cite{Tang, PonceSpector2020, MartinezSpector2021, OoiPhuc2022, ChenOoiSpector2023, ChenSpector2023, 
HH-S_JFA,HH-S_AAG, PonceSpector2023, HH-S_AWM, HH-S_La, HCYZ}.

 A new Hausdorff content maximal operator
 $\M^\delta$ by using the Choquet integral with respect to the $\delta$-dimensional Hausdorff content $\Ha^{\delta}_\infty$ in \eqref{new}  for $\delta\in (0,n]$
 was introduced  by Y.-W. Chen, K. H. Ooi, and D. Spector \cite{ChenOoiSpector2023}. 
 Boundedness results and their corollaries were  proved there.
 We continue  this study of the Hausdorff content maximal  function $\M^{\delta}f$ and its fractional counterpart, 
 and our emphasis is on functions, which are not necessarily Lebesgue measurable.

 By introducing  a  Hausdorff content Riesz potential 
$\Ri^{\delta}_{\alpha}f$  in 
\eqref{newRiesz}, using Choquet integrals with respect to 
the $\delta$-dimensional Hausdorff content $\Ha^{\delta}_\infty$,
 we give boundedness results for $\Ri^{\delta}_{\alpha}$ with $\delta\in (0,n)$ and $\Ri^n_{\alpha}$ separately.
For example, a boundedness result  for suitable functions which are not necessarily Lebesgue measurable is given in 
Theorem \ref{thm:Choquet-Hedberg}.
Our corresponding result for Lebesgue measurable functions,
Theorem \ref{thm:Choquet-Hedberg-measurable} recovers the classical result \cite[Theorem 2.8.4]{Ziemer89}.
Theorem \ref{thm:GeneralRiesz_bdd} could be seen as an extension of \cite[Theorem 2.8.4]{Ziemer89} to the Hausdorff content Riesz potential. 

In Sections \ref{sec:Hausdorff} and \ref{sec:Choquet} some basic properties of Hausdorff content and 
Choquet integral with respect to Hausdorff content are recalled.
Also, we consider a Lebesgue type space formed by non-measurable functions.
 Hausdorff content maximal operators are considered in Section~\ref{Section_Maximal}.
 Hausdorff content Riesz potential results are presented in Section~\ref{Section_RieszPotential}.

\section{Hausdorff content}\label{sec:Hausdorff}

We recall the definition of the  $\delta$-dimensional Hausdorff content for any given set $E$ 
in $\Rn$,  \cite[2.10.1, p.~169]{Federer}, \cite{Adams1998, Adams2015}
and also  its dyadic counterpart  \cite{YangYuan08}.
An  open ball centred at $x$ with radius $r>0$ is written as $B(x,r)$.

Let $E$ be a set in $\Rn$, $n \ge 1$. Suppose that
$\delta \in (0, n]$.
The $\delta$-dimensional Hausdorff content of $E$ is defined by
\begin{equation}
\Ha_\infty^{\delta} (E) := \inf \bigg\{ \sum_{i=1}^\infty r_i^{\delta}: E \subset \bigcup_{i=1}^\infty B(x_i, r_i)\bigg\}\,,\label{HausdorffC}
\end{equation}
where the infimum is taken over all  
countable (or finite) collections of balls  that cover $E$.
The quantity \eqref{HausdorffC} is also called  the $\delta$-Hausdorff capacity.

If the infimum is taken over all 
countably (or finite) collections of dyadic cubes such that the interior of the union of the cubes covers $E$, the 
dyadic counterpart of $\Ha_{\infty}^{\delta}$ is 
is given in \cite{YangYuan08}, i.e. 
\begin{equation*}
\tilde{\Ha}_\infty^{\delta} (E) := \inf \bigg\{ \sum_{i=1}^\infty \ell (Q_i)^{\delta}: E \subset \inte\big(\bigcup_{i=1}^\infty Q_i \big)\bigg\}\label{HausdorffCD}\,.
\end{equation*}
Here $\ell (Q)$ is the side length of a cube $Q$. 
The cube covering is called
a dyadic cube covering or dyadic cube cover.

Let  $\Po(\Rn)$  be a power set of $\Rn$ and $c: \Po(\Rn) \to [0, \infty]$  a set function.
The following properties are useful:
 
\begin{enumerate}
\item[(C1)] $c(\emptyset) =0$;
\item[(C2)] if $A \subset B$ then $c(A) \le c(B)$;
\item[(C3)]  if $E \subset \Rn$ then 
\[
c(E) = \inf_{E \subset U \text{ and }U \text{ is open}}c(U); 
\]
\item[(C4)] if $(A_i)$ is any sequence of sets then
\[
c\Big(\bigcup_{i=1}^\infty A_i \Big)
\le  \sum_{i=1}^\infty c(A_i);
\]
\item[(C5)] if $(K_i)$ is a decreasing sequence of compact sets then 
\[
c\Big(\bigcap_{i=1}^\infty K_i \Big)
= \lim_{i \to \infty} c(K_i);
\]
\item[(C6)] if $(E_i)$ is an increasing sequence of  sets then 
\[
c\big(\bigcup_{i=1}^\infty E_i \big)
= \lim_{i \to \infty} c(E_i).
\]
\end{enumerate}

Property (C2) is called \emph{monotonicity}  of a set function  and property (C4)  is called \emph{subadditivity}
 of a set function.
Moreover,  a set function $c$ is \emph{strongly subadditive}, if
\begin{equation}\label{st}
c(A_1 \cup A_2) + c(A_1 \cap A_2) \le c(A_1) + c(A_2)
\end{equation}
for all  subsets $A_1, A_2$ in $\Rn$. 
Property  \eqref{st}  is called submodularity  in  \cite[p. 16]{Den94}.

Results from
\cite[Theorem 2.1, Propositions 2.3 and 2.4]{YangYuan08} state that
the  $\delta$-dimensional dyadic Hausdorff content 
$\tilde{\Ha}_\infty^{\delta}$ satisfies  properties (C1), (C2), (C4), (C5), (C6),  which means that it is a Choquet capacity
 \cite{Cho53}.
 Moreover $\tilde{\Ha}_\infty^{\delta}$ is strongly subadditive.
 We  recall that there exist finite constants 
$c_1(n)>0$ and $c_2(n)>0$ such that
\begin{equation}\label{comparable}
c_1(n)\Ha_\infty^{\delta}(E)\le\tilde{\Ha}_\infty^{\delta}(E) \le c_2(n)\Ha_\infty^{\delta}(E)
\end{equation}
for any  set $E$ in $\Rn$.

On the other hand,
the $\delta$-dimensional Hausdorff content 
$\Ha_\infty^{\delta}$ satisfies properties (C1)--(C5), which means that it is an outer capacity in the sense of Meyers \cite[p. 257]{Mey70}. 
But
 $\Ha_\infty^{\delta}$ is not a capacity in the sense of Choquet \cite{Cho53},
 since it does not satisfies (C6). 
We refer to  \cite[p.~30]{Dav56} and \cite{Dav70, SioS62}.

Hence we can conclude the following remark.
\begin{remark}\label{PropertiesHausdorff}
The  $\delta$-dimensional dyadic Hausdorff content 
$\tilde{\Ha}_\infty^{\delta}$ satisfies  properties (C1), (C2), (C4), (C5), (C6), and 
$\tilde{\Ha}_\infty^{\delta}$
is strongly subadditive.
The $\delta$-dimensional Hausdorff content 
$\Ha_\infty^{\delta}$ satisfies properties (C1)--(C5).
\end{remark}

Let us recall that the $\delta$-dimensional  Hausdorff measure for a subset $E$ of  $\R^n$ is defined as
\[
\Ha^\delta (E) := \lim_{\rho \to 0^+}  \inf \bigg\{ \sum_{i=1}^\infty r_i^{\delta}: E \subset \bigcup_{i=1}^\infty B(x_i, r_i) \text{ and } r_i \le \rho \text{ for all } i\bigg\},
\]
where the infimum is taken over 
 all  countable  (or finite) balls of radius at most $\rho$ that cover $E$. 
 This covering is called a $\rho$-cover.
  As an outer measure 
 $\Ha^\delta$ is defined for all sets $E$ in $\Rn$. 
The $\delta$-dimensional Hausdorff content $\Ha^\delta_\infty$ and the $\delta$-dimensional Hausdorff measure $\Ha^\delta$ have the same null sets 
\cite[Chapter 3, p.~14]{Adams2015}.
 The $n$-dimensional Hausdorff content as a set function  is  equivalent to the $n$-dimensional Hausdorff outer measure
\cite[Proposition 2.5]{HH-S_JFA}.

We recall two results from \cite{Ponce},  namely  the value of
the $\delta$-dimensional Hausdorff content of a ball $B(x,r)$  in $\Rn$ and metric additivity.

\begin{proposition}\cite[Proposition B.2.]{Ponce}\label{lem:Antti}
Let $x\in\Rn$ and $r>0$. 
If $\delta \in (0, n]$, then
$\Ha^\delta_\infty (B(x, r)) = r^\delta$.
\end{proposition}

The proof goes as in \cite[Proposition B.2.]{Ponce}, although the definition of the Hausdorff content differs from the one we use.
We thank Antti Käenmäki for his help with Proposition~\ref{lem:Antti}.

\begin{proposition}\cite[Corollary B.13]{Ponce}
Let $A$ and $B$  be disjoint sets in $\Rn$ such that $\dist (A,B)>0$. If $\delta\in (0,n]$, then
\begin{equation*}
\Ha^\delta_\infty (A\cup B) =\Ha^\delta_\infty (A)+ \Ha^\delta_\infty (B).
\end{equation*}
\end{proposition}


\section{Choquet integral}\label{sec:Choquet}

 We recall the definition of Choquet integral and
pay attention to a variety of names which have been given to different properties  related to Choquet integrals with respect to set functions.
 Let $\Omega$ be a subset of $\Rn$, $n \ge 1$.
For any function $f:\Omega\to [0,\infty]$ the integral in the sense of Choquet with respect to Hausdorff content is defined by
\begin{equation}\label{IntegralDef}
\int_\Omega f(x) \, d H := \int_0^\infty H\big(\{x \in \Omega : f(x)>t\}\big) \, dt, 
\end{equation}
where  $H$ could be $H:=\Ha^{\delta}_\infty$ or $H:=\tilde{\Ha}_\infty^{\delta}$.
Since  $\Ha^{\delta}_\infty$  and $\tilde{\Ha}_\infty^{\delta}$ are monotone  set functions,
the corresponding distribution function
$t \mapsto H\big(\{x \in \Omega : f(x)>t\}\big)$ 
for any function $f:\Omega\to [0,\infty]$
is decreasing with respect to $t$.
By decreasing property  the  distribution function 
$t \mapsto H\big(\{x \in \Omega : f(x)>t\}\big)$ is measurable
 with respect to  the 1-dimensional  Lebesgue measure. Thus, $\int_0^\infty H\big(\{x \in \Omega : f(x)>t\}\big) \, dt$ 
 is well defined as a Lebesgue integral. 
The right-hand side of \eqref{IntegralDef} can be understood  also as an improper Riemann integral.
 The wording {\it{any function}} means that the only requirement is that the function is well defined. We emphasise that the Choquet  
 integral is well defined for any non-measurable set in $\Rn$ and
for  any non-measurable function.
A characteristic function of a non-measurable set is an example of a non-measurable function.
For an original construction of a non-measurable set by G. Vitali we refer to \cite{Vitali}.
We recall that Choquet integral is  a nonlinear integral and used in non-additive measure theory.

The Choquet integral with respect to Hausdorff content has the following properties.

\begin{lemma}\label{lem:Integral-basic-properties}
Suppose that $\Omega$ is a subset of $\Rn$,  $f, g : \Omega \to [0, \infty]$ are non-negative functions defined on $\Omega$, and $\delta \in (0,n]$.
If  $H:=\Ha^{\delta}_\infty$ or $H:=\tilde{\Ha}_\infty^{\delta}$, then
\begin{enumerate}
\item[(I1)] $ \displaystyle \int_\Omega a f(x) \, d H = a \int_\Omega  f(x) \, d H$ with any $a\ge 0$;
\item[(I2)] $\displaystyle \int_\Omega f(x) \, d H=0$ if and only if $f(x)=0$  for $H$-almost every $x\in \Omega$;
\item[(I3)] if $E\subset \Rn$, then $\displaystyle \int_\Omega \chi_E(x) \, d H = H(\Omega \cap E)$;
\item[(I4)] if $A\subset B \subset \Omega$, then $\int_A f(x) \, d H \le \int_B f(x) \, d H$;
\item[(I5)] if $0\le f\le g$, then $\displaystyle \int_\Omega f(x) \, d H\le \int_\Omega g(x) \, d H$;
\item[(I6)] $\displaystyle \int_\Omega f(x)+g(x) \, d H \le 2\Big(\int_\Omega f(x) \, d H + \int_\Omega g(x) \, d H\Big)$;
\item[(I7)] $\displaystyle \int_\Omega f(x)g(x) \, d H \le 2\Big(\int_\Omega f(x)^p \, d H\Big)^{1/p} \Big( \int_\Omega g(x)^q \, d H\Big)^{1/q}$   when   $p, q>1$ and  $\frac{1}{p}+\frac{1}{q}=1$.
\end{enumerate}
\end{lemma}

For these properties
we refer to  \cite{Adams1998}, \cite[Chapter 4]{Adams2015}, and \cite{CerMS11}. 

\begin{remark}\label{I1_I6}
(1) By property (I4)  Choquet integral is increasing with respect to the set over which the integral is taken and
by  (I5)  Choquet integral is increasing with respect to the integrand.\\
(2) Property (I6) means that  Choquet integral is quasi-subadditive and
properties
(I1) and  (I6)  imply that Choquet integral is  quasi-sublinear. \\
(3)
The H\"older inequality for Choquet integral is given by (I7).
\end{remark}

Finally, we note that for a function $f:\Omega\to [0,\infty ]$ 
\[
\int_0^\infty \Ha^{\delta}_\infty\big(\{x \in \Omega : f(x)^p>t\}\big) \, dt = \int_0^\infty p t^{p-1}\Ha^{\delta}_\infty\big(\{x \in \Omega : f(x)>t\}\big) \, dt
\]
by a change of variables.

\begin{remark}\label{lemma_a}
 There are finite constants $c_1(n)>0$ and $c_2(n)>0$ such that 
 \begin{equation*}
 c_1(n)\Ha^n_\infty (E) \le |E| \le c_2(n)\Ha^n_\infty(E)
 \end{equation*}
 for all Lebesgue measurable 
 sets $E$ in $\Rn$,  the Lebesgue measure of $E$ written as $|E|$.
 We refer to \cite[Proposition 2.5]{HH-S_JFA} and  \cite[Chapter 3, p.~14]{Adams2015} (or as well as to \cite[Chapter 2]{EvaG92}).
Hence, if  $\Omega \subset \Rn$ is Lebesgue measurable, then
 by  \cite[Proposition 2.3]{HH-S_La}
 there exist constants $c_1(n)>0$ and $c_2(n)>0$ such that
\begin{equation}\label{compare_int}
\frac{1}{c_1(n)}  \int_\Omega |f(x)|  \, d \Ha^{ n}_\infty \le \int_\Omega |f(x)| \, dx \le c_1(n)  \int_\Omega |f(x)|  \, d \Ha^{ n}_\infty
\end{equation}
for all Lebesgue measurable  functions $f:\Omega \to [-\infty, \infty]$.
These inequalities together with \eqref{comparable} yield that  the following three integrals
\[
\int_\Omega |f(x)|  \, d \Ha^{ n}_\infty, \int_\Omega |f(x)|  \, d \tilde\Ha^{ n}_\infty , \text { and } \int_\Omega |f(x)| \, dx
\]
are comparable  to each other  for Lebesgue measurable functions with  constants depending only on $n$.
\end{remark}

The following lemma can be found in \cite[Proposition 2.3]{ChenOoiSpector2023} and \cite[Proposition 2.3]{HH-S_La}.
\begin{proposition}\label{GeneralizationOV}
 If  $\Omega$ is a subset of $\Rn$ and  $0 < \delta_1 < \delta_2 \le n$, then for any function $f: \Omega \to [-\infty, \infty]$
 the  inequality
\begin{equation*}
\biggl(\int_\Omega |f(x)|\, d \Ha_{\infty}^{\delta_2} \biggr)^{1/{\delta_2}}
\le 
\biggl(
 \frac{\delta_2}{\delta_1}\biggr)^{1/\delta_2} \biggl(\int_\Omega |f(x)|^{\frac{\delta_1}{\delta_2}} \, d \Ha^{\delta_1}_\infty \biggr)^{1/\delta_1}
\label{lemma_b}
\end{equation*}
is valid.
\end{proposition}

Let us  look at  quasi-sublinearity  and sublinearity  of the Choquet integral with respect to the set function $H$,  that is
\begin{equation}\label{sublinearity}
\int_\Omega \sum_{i=1}^{\infty} f_i(x) d H \le K  \sum_{i=1}^{\infty} \int_\Omega  f_i(x) d H
\end{equation}
where $K\ge 1$ and $H=\Ha^{\delta}_\infty$ or $H=\tilde{\Ha}^{\delta}_\infty$.
Since property (I1) is valid  in cases  $H=\Ha^{\delta}_\infty$  and $H=\tilde{\Ha}^{\delta}_\infty$,
we may call property (\ref{sublinearity})
\emph{quasi-sublinearity}  of the Choquet integral whenever $K\neq 1$.
Property (\ref{sublinearity})  itself  is called \emph{sublinearity} if $K=1$.
For these names we refer to Remark \ref{I1_I6} (2).
The Choquet integral with respect to the dyadic Hausdorff content
$\tilde{\Ha}_\infty^{\delta}$ is sublinear \eqref{sublin} and   with respect to 
the classical Hausdorff content
${\Ha}_\infty^{\delta}$ it
 is quasi-sublinear  Lemma \ref{DennebergThm}.
 We point out that (C4) in Section \ref{sec:Hausdorff} of a set function itself is called a subadditivity property.

There are several result, starting from  Choquet \cite{Cho53},  Tops{\o}e \cite{Top74}, and  Anger \cite{Ang77}, where the  sublinearity
 of Choquet integral has been proved under various conditions 
 given to  functions.  However,  all these results seem to assume that the set function 
 itself is  strongly subadditive \eqref{st}. This is also known to be a necessary condition for the sublinearity of the Choquet integral
 in \cite[Chapter 6]{Den94}. 
We refer here  to a result of Denneberg \cite[Theorem 6.3, p.~75]{Den94}. Denneberg studies Choquet integrals for functions which could have  both negative
 and positive values, and thus he uses a general definition  for functions when defining Choquet integral. We study only non-negative functions and thus the definition of Denneberg coincides with our definition
for non-negative functions. Denneberg also supposes that the functions in the definition of Choquet integral 
 have so called
$\mu$-essentially $>-\infty$ property \cite[pp. 46, 74]{Den94}.
This condition  is satisfied  by non-negative functions. 
The  dyadic Hausdorff content  $\tilde{\Ha}_\infty^{\delta}$ is  a monotone, 
strongly subadditive set function.
Thus we  could apply  \cite[Theorem 6.3, p.~75]{Den94} to this Hausdorff content and obtain the following result   
for  any  $0<\delta \le n$,  
\begin{equation}\label{sublin}
\int_\Omega \sum_{i=1}^{\infty} f_i(x) d \tilde{\Ha}_\infty^{\delta} \le  \sum_{i=1}^{\infty} \int_\Omega  f_i(x) d \tilde{\Ha}_\infty^{\delta}
\end{equation}
for all non-negative functions $f_i:\Omega \to [0, \infty]$.
By inequality \eqref{comparable}  
the dyadic Hausdorff content $\tilde{\Ha}_\infty^{\delta}$ is comparable to ${\Ha}_\infty^{\delta}$, and hence
we obtain the following theorem which is crucial to  some of our results. 
Inequality \eqref{DennebergQuasiSublin} means that Choquet integral is quasi-sublinear.

\begin{theorem}\label{DennebergThm}
\cite[Theorem 6.3]{Den94}.
If $\Omega$ is a subset of $\Rn$ and  $\delta \in (0, n]$, then
for 
all sequences $(f_i)$ of non-negative functions $f_i:\Omega \to [0, \infty]$ 
\begin{equation}\label{DennebergQuasiSublin}
\int_\Omega \sum_{i=1}^{\infty} f_i(x) d \Ha^{\delta}_\infty \le c(n) \sum_{i=1}^{\infty} \int_\Omega  f_i(x) d \Ha^{\delta}_\infty\,,
\end{equation}
where the constant $c(n)$ depends only on $n$.
\end{theorem}



Next we introduce a quasi-normed space  $\L^p(\Omega, \Ha^{\delta}_\infty)$ when functions are not  necessarily Lebesgue measurable. 
 The letter ''$\textsc{n}$'' comes from  the word non-measurable.

\begin{remark}\label{LebesgueSpace}
(1) Letting $H$ be a monotone, subadditive set function
A. Ponce and D. Spector 
\cite{PonceSpector2023} 
introduce an equivalent relation $\sim$ among elements in the vector space of real-valued quasicontinuous functions
\cite[Section 6]{PonceSpector2023} in $\Rn$ by denoting $f\sim g$ whenever 
$f=g$ $H$-quasi-everywhere, that is, there exists  $E\subset\Rn$ such that $H(E)=0$ and $f=g$ in $\Rn\backslash E$. 
By \cite[Definition 3.1]{PonceSpector2023} a function $f: \R^n \to \R$ is $H$-quasicontinuous if for every $\ve>0$ there exists an open set $U\subset \Rn$ such that $H(U)\le \ve$ and $f|_{\Rn\setminus U}$ is finite and continuous in $\Rn \setminus U$.
They write that $[f]$ is the equivalence class that contains $f$ and define
\[
{L}^1(H):= \Big\{[f]: f: \R^n \to \R \mbox{ is $H$-quasicontinuous and}   \int |f|  d H < \infty \Big\}\,.
\]
If $H$ satisfies also countable subadditivity, then $L^1(H)$ is a Banach space \cite[Proposition 6.1]{PonceSpector2023}.
Assuming  more conditions for $H$  a completion result and  convergence results are proved \cite[Corollaries 6.2, 6.3, and 6.4]{PonceSpector2023}. \\
(2) If $H(E) =0$ implies $|E|=0$, then representatives of the equivalent classes in 
the space $L^1(H)$ are measurable functions by a  modification of Lusin's theorem
\cite[Theorem 4.20]{WheZ77}.
We recall that  for the Hausdorff content $\Ha^{\delta}_\infty (E)=0$  implies that $\Ha^{\delta} (E)=0$ and
also $\vert E\vert =0$ by \cite[Chapter 3, p.~14]{Adams2015}  as well as  by  \cite[Chapter 2]{EvaG92}. Thus $\Ha^\delta_\infty$-quasicontinuous functions are measurable.
\end{remark}

We consider function spaces which are formed by Choquet integrals with respect to Hausdorff content  without  assuming  quasicontinuity. 
But we are not able to obtain convergence results.
Let   $\Omega$ be a subset of $\Rn$, $n\geq 1$, $0<\delta\le n$, and $1\le p < \infty$. We write
\[
\textsc{n}\mathcal{L}^p(\Omega, \Ha^{\delta}_\infty):= \Big\{f: \Omega \to [-\infty, \infty] :   \int_\Omega |f|^p  d \Ha^{\delta}_\infty < \infty \Big\}.
\]
Properties (I1), (I5), and (I6) imply that  $\textsc{n}\mathcal{L}^p(\Omega, \Ha^{\delta}_\infty)$
is an $\R$-vector space. We define also an equivalence relation  for functions $f,g\in  
\textsc{n}\mathcal{L}^p(\Omega, \Ha^{\delta}_\infty)$
such that $f\sim g$ whenever
\[
\int_\Omega |f-g|^p  d \Ha^{\delta}_\infty=0.
\]
This means that  $f \sim g$ if and only if equality $f=g$  is valid $\Ha^{\delta}_\infty$-almost everywhere. 
We identify functions whenever they agree  
$\Ha^{\delta}_\infty$-almost everywhere and write
\begin{equation}\label{LpSpaces}
\L^p(\Omega, \Ha^{\delta}_\infty) := \textsc{n}\mathcal{L}^p(\Omega, \Ha^{\delta}_\infty)/ \sim.
\end{equation}

Recall that  a function $\|\cdot\|: X \to [0, +\infty]$ is called a \emph{quasi-norm} in an $\R$-vector space $X$ if for all $f, g \in X$ and $a \in \R$  the following three conditions are satisfied:
\begin{enumerate}
\item[(N1)] $\|f\|=0$ if and only if $f=0$;
\item[(N2)]  $\|a f\|= |a| \, \|f\|$;
\item[(N3)]  there exists $c\ge1$ such that $\|f+ g\| \le c \|f\| + c \|g\|$.
\end{enumerate}
If  $c=1$ in (N3), then $\|\cdot\|$ is a \emph{norm}.

Choquet integral gives a quasi-norm. The constant $2$ in the left-hand side of  equalities (I6) ja (I7) means that a norm is not possible  in this context.

\begin{proposition}\label{thm:quasinorm}
If $\Omega$ is a subset of  $\Rn$,  $n\geq 1$, $0<\delta\le n$, and $1\le p < \infty$,  then
\[
\|f\|_{\L^p(\Omega, \Ha^{\delta}_\infty)}:= \Big( \int_\Omega|f|^p d \Ha^{\delta}_\infty  \Big )^{\frac 1p}
\]
is a quasi-norm  in $\L^p(\Omega, \Ha^{\delta}_\infty)$. 
\end{proposition}

\begin{proof}
Clearly $\|0\|_{\L^p(\Omega, \Ha^{\delta}_\infty)}=0$.
If
$\|f\|_{\L^p(\Omega, \Ha^{\delta}_\infty)}=0$,
then by the definition $\Ha^{\delta}_\infty(\{x \in \Omega : |f(x)|^p >t\})=0$ for every $t>0$. Thus for $\Ha^{\delta}_\infty$-almost everywhere
$f=0$ and  condition (N1) holds.
By property  (I1) condition (N2) is clear.

Let us prove (N3) and assume for that $f, g \in \L^p(\Omega, \Ha^{\delta}_\infty)$.
The case $p=1$ follows form (I6). Let $1<p<\infty$.
We may assume that 
$\|f+ g\|_{\L^p(\Omega, \Ha^{\delta}_\infty)}>0$. Using the Euclidean triangle inequality pointwise implies that
 $|(f+g)(x)|^p \le |(f+g)(x)|^{p-1}(|f(x)| + |g|(x))$.
Hence, properties 
(I5), (I6), and Hölder's inequality (I7) imply that
\[
\begin{split}
\int_\Omega |f+g|^p \,d \Ha^{\delta}_\infty
& \le 2 \int_\Omega |f+g|^{p-1} |f|  \,d \Ha^{\delta}_\infty + 2\int_U |f+g|^{p-1} |g|  \,d  \Ha^{\delta}_\infty\\
& \le 4 \Big(\int_\Omega |f+g|^{(p-1)p'} \,d \Ha^{\delta}_\infty\Big)^{\frac1{p'}} \Big(\int_U |f|^p  \,d \Ha^{\delta}_\infty \Big)^{\frac1p}  \\
& \quad + 4 \Big(\int_\Omega |f+g|^{(p-1)p'} \,d \Ha^{\delta}_\infty\Big)^{\frac1{p'}} \Big(\int_U |g|^p  \,d \Ha^{\delta}_\infty \Big)^{\frac1p}.
\end{split}
\]
Since $\Big(\int_\Omega |f+g|^{(p-1)p'} \,d \Ha^{\delta}_\infty\Big)^{\frac1{p'}} = \Big(\int_\Omega |f+g|^{p}\,d  \Ha^{\delta}_\infty\Big)^{\frac1{p'}}>0$ the claim follows.
\end{proof}

\section{On Hausdorff content maximal operator}\label{Section_Maximal}

The  Hausdorff content centred maximal function was defined  by  Y.-W. Chen, K. H. Oii, and D. Spector in \cite{ChenOoiSpector2023}.
We recall the definition for functions which are not necessarily Lebesgue measurable.

Let $f:\Rn \to [-\infty, \infty]$, $\delta \in(0, n]$, and $\kappa \in[0, \delta)$. 
The   Hausdorff content  centred fractional maximal function is defined as
\begin{equation}\label{new}
\M^\delta_\kappa f(x) := \sup_{r>0} \frac{r^\kappa}{\Ha_\infty^\delta (B(x, r))} \int_{B(x, r)} |f(y)| \, d \Ha_\infty^\delta(y).
\end{equation}
We write  $\M_0^\delta =: \M^\delta$.

We recall that for  Lebesgue  measurable functions $f:\Rn \to [-\infty, \infty]$ the classical centred fractional Hardy-Littlewood maximal operator is defined by
\begin{equation}\label{classical}
M_\kappa f(x) := \sup_{r>0} \frac{r^\kappa}{|B(x, r)|} \int_{B(x, r)} |f(y)| \, dy.
\end{equation}
By Remark \ref{lemma_a} there exist  finite constants $c_1=c_1(n)>0$ and $c_2=c_2(n)>0$ such that
\begin{equation}\label{compareMM^n}
c_1 \M^n_\kappa f(x) \le M_\kappa f(x) \le c_2 \M^n_\kappa f(x)
\end{equation}
for all  Lebesgue measurable $f: \Rn \to [-\infty, \infty]$, i.e.\ for measurable functions 
the  Hausdorff content  centred fractional maximal function
$\M^n_\kappa f$ 
is comparable with 
the classical centred fractional  Hardy--Littlewood maximal function $M_\kappa f$.
But we are  interested in functions which are not  necessarily Lebesgue measurable.

D.\ R.\ Adams  showed that
the classical Hardy-Littlewood maximal operator  \eqref{classical} satisfies a strong type inequality in the sense of Choquet with respect to the Hausdorff content
\cite[Theorem A]{Adams1986}, \cite[Theorem 7. (a)]{Adams1998}.
We refer also to \cite[Theorem (i)]{OV} by
J.\ Orobitg and J.\ Verdera. 
 Y.-W. B. Chen, K. H. Ooi, and D. Spector showed the boundedness for the   Hausdorff content centred maximal function  $\M^{\delta}f$
\cite[Theorem 1.2]{ChenOoiSpector2023} when  $\delta \in (0,n)$. The case 
for $\M^n f$ is  considered in \cite[Corollary 1.4]{ChenOoiSpector2023}.  We will complement this last result.

\begin{remark} The following 
 Proposition~\ref{prop:lsc} shows 
 that the maximal function $\M^\delta_\kappa f$ is Lebesgue measurable, although the function $f$ itself might be non-measurable.
\end{remark}

Properties  (I1) and (I7)  imply that the fractional maximal operator $\M^\delta_\kappa$ is quasi-sublinear. Recall the definition from Remark \ref{I1_I6}. We show that the fractional maximal function 
$\M^\delta_\kappa f$
 is lower semicontinuous.

\begin{proposition}\label{prop:lsc}
If $n \ge 1$, $\delta \in (0, n]$, $\kappa \in [0, \delta)$,  and $f: \Rn \to [-\infty, \infty]$  any function,
then the function  $\M^\delta_\kappa f$ is lower semicontinuous.
\end{proposition}

\begin{proof}
We have to show that $\{x \in \Rn : \M^\delta_\kappa f(x) >t\}$ is an open set for every $t\in \R$.  Since $\M^\delta_\kappa f$ is non-negative, we may assume that $t\ge 0$. 
Let us pick a point  $x_0 \in \{x \in \Rn : \M^\delta_\kappa f(x) >t\}$ and choose 
$R>0$ so that
\[
\frac{R^\kappa}{\Ha_\infty^\delta (B(x_0, R))} \int_{B(x_0, R)} |f(x)| \, d \Ha_\infty^\delta(x)>t.
\]
Let us then pick up   $\ve >0$ and choose a point  $y \in B(x_0, \ve)$. 
Then  $  B(x_0, R)\subset B(y, R +\ve)$, and hence  the monotonicity of the Choquet integral  (I4) implies that
\[
\int_{B(x_0, R)} |f(x)| \, d \Ha_\infty^\delta \le \int_{B(y, R+\ve)} |f(x)| \, d \Ha_\infty^\delta.
\]
By the monotonicity of the Hausdorff content  (C2), Remark \ref{PropertiesHausdorff}, and Proposition~\ref{lem:Antti} we obtain
\[
0\le \Ha_\infty^\delta (B(y, R+\ve))- \Ha_\infty^\delta (B(x, R)) = (R+\ve)^\delta - R^\delta \to 0
\]
as $\ve \to 0^+$.
Thus we may choose $\ve>0$ to be so small that 
\[
\frac{(R+\ve)^\kappa}{\Ha_\infty^\delta (B(y, R+\ve))} \int_{B(y, R+\ve)} |f(x)| \, d \Ha_\infty^\delta(x) >t
\]
for all $y \in B(x_0, \ve)$. 
This yields that $B(x_0, \ve) \subset \{x \in \Rn : \M^\delta_\kappa f(x) >t\}$ i.e. $\{x \in \Rn : \M^\delta_\kappa f(x) >t\}$ is an open set. 
\end{proof}

It was shown in \cite[Theorem 1.2 and Corollary 1.4]{ChenOoiSpector2023} that the  Hausdorff content centred maximal function
 \eqref{new} is bounded.   
 We complement the boundedness result for $\M^nf$.
This  result extends \cite[Theorem A]{Adams1986}, 
\cite[Theorem (i) ]{OV},   and \cite[Corollary 1.4(i)]{ChenOoiSpector2023}  to   cover  also  other than Lebesgue measurable functions.

\begin{theorem}\label{thm:M-bounded}
Let  $n \ge 1$, $\delta\in(0, n]$ and $p \in (\delta/n, \infty)$.
Then there exists a constant $c$ depending only on $n$, $\delta$, and $p$ such that
\begin{equation}\label{strongineq}
\int_\Rn (\M^n f(x))^p \, d \Ha^{\delta}_\infty \le c \int_\Rn |f(x)|^p \, d \Ha^{\delta}_\infty
\end{equation}
for all functions  $f:\Rn \to [-\infty, \infty]$.
\end{theorem}

\begin{remark}
(1) Functions in Theorem~\ref{thm:M-bounded} are not necessarily Lebesgue measurable, but functions $\M^n f$ are Lebesgue measurable
by Proposition~\ref{prop:lsc}.\\
(2) From  estimate \eqref{lemma-b} in the proof of this theorem  can be seen  that the constant  $c$ in \eqref{strongineq} blows up as $p \to \delta/n$.
\end{remark}

\begin{proof}[Proof of Theorem \ref{thm:M-bounded}]
We may assume that the right-hand side of \eqref{strongineq}  is finite and $f \ge 0$. Then $0 \le f(x)<\infty$ for $\Ha^{\delta}_\infty$-almost every $x \in \Rn$.

Assume first that $\delta\in(0, n)$.
We follow the proof of \cite[Theorem]{OV}.  
Let us take any function  $f:\Rn \to [0, \infty]$ and fix $k \in \Z$. We study   the set $\{x \in \Rn : 2^{k} < f(x) \le 2^{k+1}\}$. 
If $\Ha_\infty^{\delta} \big(\{x \in \Rn : 2^{k} < f(x) \le 2^{k+1}\} \big) >0$, 
 by
the definition of $\Ha_\infty^\delta$  we  find
balls  $B_j^k:= B(x_j, r_j)$, $j=1,2, \ldots, $ such that
\[
\{x \in \Rn : 2^{k} < f(x) \le 2^{k+1}\} \subset \bigcup_{j=1}^\infty B_{j}^k 
\]
and 
\[ 
\sum_{j=1}^\infty r_{j}^\delta \le 2  \Ha_\infty^{\delta} \big(\{x \in \Rn : 2^{k} < f(x) \le 2^{k+1}\} \big).
\]
 If $\Ha_\infty^{\delta} \big(\{x \in \Rn : 2^{k} < f(x) \le 2^{k+1}\} \big)=0$, no ball is chosen.
Now we consider the case $p \geq 1$ first.
We set $A_k := \bigcup_j B_{j}^k$ and $g := \sum_k 2^{p(k+1)} \chi_{A_k}$, and note that $A_k$ is  a  Lebesgue measurable set  and $g$ 
 is a  Lebesgue measurable function.
Moreover,   the inequality $f(x)^p \le g(x)$  holds pointwise
$\Ha^{\delta}_\infty$-almost everywhere, since  the set $\{x: f(x)=\infty\}$ is not covered by sets $A_k$. 
In the case $p>1$ by  the H\"older inequality for Choquet integrals, that is the aforementioned property (I7), we have
\[
\begin{split}
&\frac{1}{\Ha_\infty^{n} (B(x, r))} \int_{B(x, r)} |f(y)| \, d \Ha_\infty^{n}(y) \\
&\le  \frac{1}{\Ha_\infty^{n} (B(x, r))} 2 \Big( \int_{B(x, r)} |f(y)|^p \, d \Ha_\infty^n(y)\Big)^{\frac1p} \Big( \int_{B(x, r)} 1^{p'} \, d \Ha_\infty^n(y)\Big)^{\frac1{p'}} \\
&\le  \frac{1}{\Ha_\infty^n (B(x, r))^{1-\frac1{p'}}} 2 \Big( \int_{B(x, r)} |f(y)|^p \, d \Ha_\infty^n(y)\Big)^{\frac1p}\\
&=  2 \bigg( \frac{1}{\Ha_\infty^n (B(x, r))} \int_{B(x, r)} |f(y)|^p \, d \Ha_\infty^n(y)\bigg)^{\frac1p}\,.
\end{split}
\]
Whenever $p=1$, the above estimate is trivial.
Thus
\[
(\M^n f(x))^p \le  2^p \M^n (f^p)(x) \le 2^p \M^n g(x)
\]
 for $\Ha_\infty^\delta$-almost all $x$. The latter inequality  is true, since 
  $f^p(x) \le g$(x)  holds $\Ha^\delta_\infty$-almost everywhere.
Since $g$ is Lebesgue  measurable, 
we obtain by using 
the classical Hardy--Littlewood maximal operator 
\[
\M^n g(x) \le c(n) M g(x) \le  c(n) \sum_{k=-\infty}^\infty 2^{p(k+1)} \sum_{j=1}^\infty M  \chi_{B_{j}^k} (x).
\]
Monotonicity (I4) and 
quasi-linearity of  Choquet integral Theorem~\ref{DennebergThm} imply that
\begin{equation}\label{ensimm}
\begin{split}
\int_\Rn (\M^n f(x))^p \, d \Ha_\infty^{\delta}  &\le 2^p  \int_\Rn \M^n g(x) \, d \Ha_\infty^{\delta}\\
&\le  c(n, p)  \sum_{k=-\infty}^\infty 2^{p(k+1)} \sum_{j=1}^\infty \int_\Rn M 
\chi_{B_{j}^k} (x) \, d \Ha_\infty^{\delta}.
\end{split}
\end{equation}
Next we use \cite[Lemma 1]{OV}  to estimate the last integral
 in \eqref{ensimm}. Since  the proof 
of  \cite[Lemma 1]{OV} is short, we include it here.  
First we estimate
\[
M \chi_{B_{j}^k}(x) \le c(n) \inf \bigg\{1, \frac{r_{j}^n}{|x-x_{j}|^n} \bigg \}\,.
\]
Thus
\[
\begin{split}
\int_\Rn M \chi_{B_{j}^k}(x)\, d \Ha_\infty^{\delta} 
&\le 2\int_{2 B_{j}^k} c(n) \, d \Ha_\infty^{\delta}  + 2\int_{\Rn\setminus 2B_{j}^k}    \frac{c(n) r_{j}^{n}}{|x-x_{j}|^{n}} \, d \Ha_\infty^{\delta} \\ 
&\le c(n) r_{j}^{\delta} + c(n) \int_0^1   \Ha_\infty^{\delta} \big(\{x \in \Rn : r_{j}^{n} |x-x_{j}|^{-n}>t\}\big) \, dt.
\end{split}
\]
We note that $\{x \in \Rn : r_{j}^{n} |x-x_{j}|^{-n}>t\} = B(x_{j}, r)$, where
$r= r_{j} t^{-1/n}$.  Since $\delta<n$, we obtain
\begin{equation}\label{lemma}
\begin{split}
\int_\Rn M \chi_{B_{j}^k}(x) \, d \Ha_\infty^{\delta} (x) 
&\le c(n) r_{j}^{\delta} + c(n)  \int_0^1 r_{j}^{\delta}  t^{-\frac{\delta}{n}} dt
 \le \frac{c(n)}{n-\delta} r_{j}^{\delta}\,.
\end{split}
\end{equation}
Now  we combine  \eqref{ensimm} and \eqref{lemma}  to estimate
\[
\begin{split}
&\int_\Rn (\M^n f(x))^p \, d \Ha_\infty^{\delta} 
\le  c(n, \delta, p) \sum_{k=-\infty}^\infty 2^{p(k+1)} \sum_{j=1}^\infty r_{j}^\delta\\
&\qquad\le  c(n, \delta, p)  \sum_{k=-\infty}^\infty 2^{p(k+1)+1}  \Ha_\infty^{\delta} \big(\{x \in \Rn : 2^{k} < f(x) \le 2^{k+1}\} \big)\\
&\qquad\le  c(n, \delta, p)  \sum_{k=-\infty}^\infty  \frac{2^{2p}}{2^p -1}\int_{2^{p(k-1)}}^{2^{pk}}  \Ha_\infty^{\delta} \big(\{x \in \Rn : 2^{k} < f(x) \le 2^{k+1}\} \big) \, dt\\
&\qquad\le  c(n, \delta, p)  \sum_{k=-\infty}^\infty  \frac{2^{2p}}{2^p -1}\int_{2^{p(k-1)}}^{2^{pk}}  \Ha_\infty^{\delta} \big(\{x \in \Rn :  (f(x))^p > 2^{pk}\} \big) \, dt\\
&\qquad\le  c(n, \delta, p)  \sum_{k=-\infty}^\infty \int_{2^{p(k-1)}}^{2^{pk}}  \Ha_\infty^{\delta} \big(\{x \in \Rn :  (f(x))^p > t\} \big) \, dt.
\end{split}
\]
Next we use the fact that Lebesgue integral over essentially  disjoint sets can be calculate as an integral over an union of the sets and obtain
\[
\begin{split}
\int_\Rn (\M^n f(x))^p \, d \Ha_\infty^{\delta}  
&\le  c(n, \delta, p) \int_{0}^{\infty}  \Ha_\infty^{\delta} \big(\{x \in \Rn :  (f(x))^p > t\} \big) \, dt\\
&= c(n, \delta, p)\int_{0}^{\infty} (f(x))^p \, d \Ha_\infty^{\delta}.
\end{split}
\]
This yields the claim in the case $p \ge 1$.

Assume then that $\delta/n<p <1$.  Let us use sets $A_k$ from the previous case and write  $h := \sum_k 2^{k+1} \chi_{A_k}$. Since $h$ is Lebesgue  measurable we obtain by using the classical Hardy--Littlewood maximal operator 
\[
\M^n f(x) \le c(n) M h(x) \le  c(n) \sum_{k=-\infty}^\infty 2^{k+1} \sum_{j=1}^\infty M  \chi_{B_{j}^k} (x).
\]
Since $p<1$,  we have
\[
(\M^n f(x))^p \le c(n)  \sum_{k=-\infty}^\infty 2^{p(k+1)} \sum_{j=1}^\infty  (M  \chi_{B_{j}^k} (x))^p.
\]
Thus, Lemma~\ref{DennebergThm} yields that
\[
\begin{split}
\int_\Rn (\M^n f(x))^p\, d \Ha_\infty^{\delta}
&\le c(n) \sum_{k=-\infty}^\infty 2^{p(k+1)} \sum_{j=1}^\infty  \int_\Rn (M  \chi_{B_{j}^k} (x))^p \, d \Ha_\infty^{\delta}\,.
\end{split}
\]
We use the same estimate for the Hardy-Littlewood maximal function as in the previous case and obtain
\[
\begin{split}
&\int_\Rn (M \chi_{B_{j}^k}(x))^p\, d \Ha_\infty^{\delta} (x) 
\le 2\int_{2B_{j}^k} c(n) \, d \Ha_\infty^{\delta} (x)  + 2\int_{\Rn\setminus 2B_{j}^k}    \frac{c(n) r_{j}^{pn}}{|x-x_{j}|^{pn}} \, d \Ha_\infty^{\delta} (x)\\ 
&\quad\le c(n, \delta) r_{j}^{\delta} + c(n) \int_0^1   \Ha_\infty^{\delta} \big(\{x \in \Rn : r_{j}^{pn} |x-x_{j}|^{-pn}>t\}\big) \, dt.
\end{split}
\]
We note that $\{x \in \Rn : r_{j}^{pn} |x-x_{j}|^{-pn}>t\} = B(x_{j}, r)$, where
$r= r_{j} t^{-1/pn}$.  Since $\delta/n<p$, we have
\begin{equation}\label{lemma-b}
\begin{split}
\int_\Rn M \chi_{B_{j}^k}(x) \, d \Ha_\infty^{\delta} (x) 
&\le c(n) r_{j}^{\delta} + c(n)  \int_0^1 r_{j}^{\delta}  t^{-\frac{\delta}{pn}} dt\\
& \le \frac{c(n)}{pn-\delta} r_{j}^{\delta}\,.
\end{split}
\end{equation}
Hence,  as in the previous case 
\[
\begin{split}
\int_\Rn (\M^n f(x))^p \, d \Ha_\infty^{\delta} 
&\le  c(n, \delta)  \sum_{k=-\infty}^\infty 2^{p(k+1)} \sum_{j=1}^\infty r_{j}^\delta\\
&\le c(n, \delta, p)\int_{0}^{\infty} (f(x))^p \, d \Ha_\infty^{\delta}\,.
\end{split}
\]
Thus the claim follows also when $p \in (\delta/n, 1)$.

Assume then that $\delta=n$ and $p \in (1, \infty)$.
Let  the sets $A_k$ and  the function $h$  be as before, i.e.\  $A_k = \bigcup_j B_{j}^k$  and  $h = \sum_k 2^{k+1} \chi_{A_k}$. Then, the function  $h$ is Lebesgue  measurable, and
$f(x) \le h(x) \le 2 f(x)$ for all  points  $x \in \Rn \setminus\{x: f(x)=\infty\}$. Recall  that $\Ha^{n}_\infty(\{x: f(x) =\infty\})=0$, since we 
assumed the right-hand side of \eqref{strongineq}   to be finite.
Estimates \eqref{compare_int}, \eqref{compareMM^n},  and the boundedness of the classical Hardy-Littlewood maximal operator imply that 

 \begin{align*}
 \int_\Rn (\M^n f(x))^p \, d \Ha_\infty^{n}
&\le \int_\Rn (\M^n h(x))^p \, d \Ha_\infty^{\delta}
\le c(n) \int_\Rn (M h(x))^p \, dx\\
&\le \frac{c(n) p}{p-1} \int_\Rn h(x)^p \, dx
\le c(n, p) 2^p \int_\Rn f(x)^p \, dx\\
&\le c(n, p) \int_\Rn f(x)^p \, d \Ha_\infty^{n}.  \qedhere
 \end{align*}
\end{proof}

\begin{remark}
We emphasise that in Theorem~\ref{thm:M-bounded} only the  operator $\M^n$ is studied.
The operator  $\M^\delta$ with $\delta<n$ is  not included, since  in the proof  of 
Theorem~\ref{thm:M-bounded} the term
$(\M^n f(x))^p$ is estimated by 
$M g(x)$ which  works only for $\M^n$.  
\end{remark}

Next we study the  Hausdorff content centred fractional  maximal function. It seems that there are  no previous results for this 
Hausdorff content  fractional  maximal operator.

\begin{proposition}\label{pointwise}
If $n \ge 1$, $d \in (0, n]$, $\kappa \in (0, d)$, and $1<q< \frac{d}{\kappa}$, then for any function $f:\Rn \to [-\infty, \infty]$
\begin{equation}\label{proposition_pointwise}
\M^d_\kappa f(x) \le c \Big(\int_{\Rn} |f(y)|^{q} \, d \Ha^{d}_\infty (y) \Big)^{\frac{\kappa}{d}} 
\big( \M^d f(x)\big)^{1-\frac{q\kappa}{d}} \mbox{ for all } x\in \Rn,
\end{equation}
where $c$ is a constant which depends on $d$.
\end{proposition}

\begin{proof}
Using Lemma~\ref{lem:Antti} and twice Hölder's inequality (I7) imply that
\[
\begin{split}
&\frac{r^\kappa}{ \Ha^{d}_\infty (B(x,r))} \int_{B(x, r)} |f(y)| \, d \Ha^{d}_\infty (y)
\\
&\le 2 \Ha^{d}_\infty (B(x,r))^{-1+\frac{\kappa}{d}} \Big( \int_{B(x, r)} |f(y)| \, d \Ha^{d}_\infty (y)\Big)^{\frac{q\kappa}{d}} \Big( \int_{B(x, r)} |f(y)| \, d \Ha^{n}_\infty (y)\Big)^{1-\frac{q\kappa}{d}} \\
&\le 2\Ha^{d}_\infty (B(x,r))^{-1+\frac{\kappa}{d}} \bigg( 2\Ha^{d}_\infty (B(x,r))^{\frac{1}{q'}} \Big(\int_{B(x, r)} |f(y)|^q \, d \Ha^{d}_\infty (y) \Big)^{\frac1q} \bigg)^{\frac{q\kappa}{d}}\\
&\qquad \cdot  \Big( \int_{B(x, r)} |f(y)| \, d \Ha^{d}_\infty (y)\Big)^{1-\frac{q\kappa}{d}} \\
&\le c(d) \Big(\int_{\Rn} |f(y)|^{q} \, d \Ha^{d}_\infty (y) \Big)^{\frac{\kappa}{d}} \Big( \Ha^{d}_\infty (B(x,r))^{-1} \int_{B(x, r)} |f(y)| \, d \Ha^{d}_\infty (y)\Big)^{1-\frac{q\kappa}{d}}.
\end{split}
\]
Taking supremum over all radii  $r$   gives  inequality \eqref{proposition_pointwise} for all $x$.
\end{proof}

Theorem~\ref{thm:M-bounded} and Proposition~\ref{pointwise} give the following result.

\begin{theorem}\label{thm:fractional-M-bounded}
Let  $n \ge 1$, $\delta\in(0, n]$, $\kappa \in [0, \delta)$, and $p \in (\delta/n, \delta/\kappa)$.
Then there exists a constant $c$ depending only on $n$, $\delta$, $\kappa$, and $p$ such that
\begin{equation}\label{strongineq1b}
\bigg(\int_\Rn (\M^n_\kappa f(x))^{\frac{\delta p}{\delta - p\kappa}} \, d \Ha^{\delta}_\infty \bigg)^{\frac{\delta - p\kappa}{\delta p}} \le c \bigg(\int_\Rn |f(x)|^p \, d \Ha^{\delta}_\infty\bigg)^{\frac1p}
\end{equation}
for all functions  $f:\Rn \to [-\infty, \infty]$.
\end{theorem}

\begin{proof}
If $\kappa =0$, the claim follows straight from Theorem~\ref{thm:M-bounded}. We may assume that $\kappa >0$.
Let us  choose $d=n$ in Proposition \ref{pointwise}. By Theorem~\ref{thm:M-bounded} and Proposition~\ref{GeneralizationOV}  we obtain
\[
\begin{split}
\int_\Rn \big( \M^n_\kappa f(x) \big)^{\frac{n p}{n - q\kappa}}\, d \Ha^{\delta}_\infty &
\le c(n) \Big(\int_{\Rn} |f(y)|^{q} \, d \Ha^{n}_\infty \Big)^{\frac{ \kappa p}{n - q \kappa}}
\int_\Rn \big( \M^n f(x) \big)^{p}\, d \Ha^{\delta}_\infty \\
&\le c(n, \delta, p) \Big(\int_{\Rn} |f(y)|^{\frac{q\delta}{n}} \, d \Ha^{\delta}_\infty \Big)^{\frac{n \kappa p}{\delta(n - q \kappa)}}
\int_\Rn  |f(x)|^{p}\, d \Ha^{\delta}_\infty. 
\end{split}
\]
Let us write $q:= p \frac{n}{\delta}$.  Note that  the assumption $p \in (\delta/n, \delta/\kappa)$
yields that $1<q< \frac{n}{\kappa}$. Thus,  we have
\[
\int_\Rn \big( \M^n_\kappa f(x) \big)^{\frac{\delta p}{\delta - p\kappa}}\,d \Ha^{\delta}_\infty 
\le c
\Big(\int_{\Rn} |f(x)|^{p} \, d \Ha^{\delta}_\infty \Big)^{\frac{ \delta}{\delta - p\kappa}},
\]
which gives the claim.
\end{proof}

Let us study $\M^\delta_\kappa$, when  $d<n$. We are able to  use 
\cite[Theorem 1.2]{ChenOoiSpector2023}, if we assume that the functions are quasicontinuous.

\begin{theorem}\label{thm:fractional-M-bounded-2}
Let  $n \ge 1$, $\delta\in(0, n)$, $\kappa \in [0, \delta)$, and $p \in (1, \delta/\kappa)$.
Then there exists a constant $c$ depending only on $n$, $\delta$, $\kappa$, and $p$ such that
\begin{equation*}
\bigg(\int_\Rn (\M^\delta_\kappa f(x))^{\frac{\delta p}{\delta - p\kappa}} \, d \Ha^{\delta}_\infty \bigg)^{\frac{\delta - p\kappa}{\delta p}} \le c \bigg(\int_\Rn |f(x)|^p \, d \Ha^{\delta}_\infty\bigg)^{\frac1p}
\end{equation*}
for all $\Ha^\delta_\infty$-quasicontinuous function   $f : \Rn \to [- \infty, \infty]$.
\end{theorem}

\begin{proof}
In the case $\kappa=0$ the theorem is exactly  \cite[Theorem 1.2 (i)]{ChenOoiSpector2023}.
We may assume that $\kappa >0$. Let us choose in Proposition \ref{pointwise} that $d= \delta$ and $q=p$. Proposition~\ref{pointwise}
 and   \cite[Theorem 1.2 (i)]{ChenOoiSpector2023} yield that
\begin{align*}
\int_\Rn \big( \M^\delta_\kappa f(x) \big)^{\frac{\delta p}{\delta - p\kappa}}\, d \Ha^{\delta}_\infty  &\le c(\delta) \Big(\int_{\Rn} |f(y)|^{p} \, d \Ha^{\delta}_\infty  \Big)^{\frac{ \kappa p}{\delta - p \kappa}}
\int_\Rn \big( \M^\delta f(x) \big)^{p}\, d \Ha^{\delta}_\infty \\
&\le c \Big(\int_{\Rn} |f(y)|^{p} \, d \Ha^{\delta}_\infty  \Big)^{\frac{ \kappa p}{\delta - p \kappa}}
\int_\Rn  |f(x)|^{p}\, d \Ha^{\delta}_\infty\\
&\le c \Big(\int_{\Rn} |f(y)|^{p} \, d \Ha^{\delta}_\infty \Big)^{\frac{\delta}{\delta- p \kappa}}. \qedhere
\end{align*}
\end{proof}


The   Hausdorff content uncentered fractional maximal operator can be defined in an analogous way:
\begin{equation*}
\Mu^\delta_\kappa f(x) := \sup_{B_r \ni x} \frac{r^\kappa}{\Ha_\infty^\delta (B_r)} \int_{B_r} |f(y)| \, d \Ha_\infty^\delta(y),
\end{equation*}
where the supremum is taken over all open balls $B_r$, with radius $r$, containing $x$.
Lower semicontinuity of  the function $\Mu^\delta_\kappa f :\Rn\to [0,\infty]$ follows now by the definition.
Thus, the function $\Mu^\delta_\kappa f$ is Lebesgue measurable for all functions $f: \Rn \to [- \infty, \infty]$.

Moreover, these maximal operators are comparable, 
that is
\begin{equation}\label{comp}
\M^\delta_\kappa f(x) \le \Mu^\delta_\kappa f(x) \le 2^{\delta- \kappa} \M^\delta_\kappa f(x) \mbox{ for all } x \in \Rn.
\end{equation}
Namely,
let us fix a point $x \in \Rn$, and let $\ve >0$. By the definition of the maximal operator there exists  an open ball
$B(y, r)$  such that  $x \in B(y, r)$ and 
\[
 \Mu^\delta_\kappa f(x) \le \frac{r^\kappa}{\Ha_\infty^\delta (B(y, r))} \int_{B(y, r)} |f(z)| \, d \Ha_\infty^\delta(z) + \ve\,.
\]
Then,  by Lemma~\ref{lem:Antti} 
\[
\begin{split}
\Mu^\delta_\kappa f(x) &\le  \frac{r^\kappa}{\Ha_\infty^\delta (B(y,r))} \int_{B(x, 2r)} |f(z)| \, d \Ha_\infty^\delta(z) + \ve\\
&\le  2^{-\kappa} \frac{\Ha_\infty^\delta (B(x,2r))}{\Ha_\infty^\delta (B(y,r))} \M^\delta_\kappa f(x) + \ve
\le 2^{\delta- \kappa}  \M^\delta_\kappa f(x) +\ve. 
\end{split}
\]
Letting $\ve \to 0^+$, we obtain $\Mu^\delta_\kappa f(x)\le 2^{\delta- \kappa}\M^\delta_\kappa f(x)$,  and this for all $x\in\Rn$.
The inequality $\M^\delta_\kappa f(x) \le \Mu^\delta_\kappa f(x)$ follows from the definition.
Because of inequality \eqref{comp} Theorem~\ref{thm:fractional-M-bounded} yields the following corollary.

\begin{corollary}\label{cor:tildeM-bounded}
Let  $n \ge 1$, $\delta\in(0, n]$, $\kappa \in [0, \delta)$ and $p \in (\delta/n, \delta/\kappa)$.
Then there exists a constant $c$ depending only on $n$, $\delta$, $\kappa$, and $p$ such that
\begin{equation*}
\bigg(\int_\Rn (\Mu^n_\kappa f(x))^{\frac{\delta p}{\delta - p\kappa}} \, d \Ha^{\delta}_\infty \bigg)^{\frac{\delta - p\kappa}{\delta p}} \le c \bigg(\int_\Rn |f(x)|^p \, d \Ha^{\delta}_\infty\bigg)^{\frac1p}
\end{equation*}
for all functions  $f:\Rn \to [-\infty, \infty]$.
\end{corollary}


We define Hausdorff content sharp maximal function.
Let $f: \Rn \to [- \infty, \infty]$ be a function and let $B$ be an open ball in $\Rn$. If $\delta \in (0,n]$ and $f$ is 
locally integrable with respect to the $\delta$-dimensional Hausdorff content, we write for the integral  average in the sense of Choquet integrals that
\begin{equation*}
f_{B,\delta} :=  \frac{1}{\Ha^{\delta}_\infty (B)} \int_B |f(y)| \, d\Ha^{\delta}_\infty
\end{equation*}
and define the Hausdorff content sharp maximal function by setting
\begin{equation*}
\M^{\#}_{\delta} f(x) :=  \sup_{B \ni x} \frac{1}{\Ha^{\delta}_\infty (B)} \int_B |f (y)-f_{B,\delta} |\, d \Ha^{\delta}_\infty,
\end{equation*}
where the supremum is taken over all open balls containing $x$.
 If $\delta =n$, we write shortly
 $f_{B,\delta} =:f_B$ and $\M^{\#}_{n} f(x)=:\M^{\#} f(x)$.
This sharp maximal function is comparable to the Fefferman--Stein  sharp maximal function for Lebesgue measurable functions.

\begin{proposition}
Suppose that $n \ge 1$,  $\delta\in(0, n]$ and $p \in (\delta/n, \infty)$.
Then, for all functions  $f\in\L^p(\Rn, \Ha^{\delta}_\infty)$
\begin{equation*}
\int_\Rn (\M^{\#} f(x))^p \, d \Ha^{\delta}_\infty \le c \int_\Rn |f(x)|^p \, d \Ha^{\delta}_\infty\,,
\end{equation*}
where a constant $c$ depends only on $n$, $\delta$, and $p$.
\end{proposition}

\begin{proof} By the definition  and inequality (I6)  we obtain the estimate
\[
\M^{\#} f(x) \le  \sup_{B \ni x} \frac{2}{\Ha^n_\infty (B)} \int_B |f(y)| \, d\Ha^n_\infty + \sup_{B \ni x} \frac{2}{\Ha^n_\infty (B)} \int_B f_B  \, d\Ha^n_\infty
\le  4 \Mu^n f(x).
\]
 The claim follows now  by Corollary~\ref{cor:tildeM-bounded}.
\end{proof}


\section{Mapping properties of Hausdorff content Riesz potential}\label{Section_RieszPotential}

Suppose that $n \ge 1$, $\delta \in (0, n]$ and $\alpha \in(0, \delta)$.
We define the  Hausdorff content Riesz potential  of a  given function $f$  as a Choquet integral with respect to the $\delta$-dimensional  Hausdorff content
as follows
\begin{equation}\label{newRiesz}
\Ri_\alpha^\delta f(x) := \int_{\Rn} \frac{|f(y)|}{|x-y|^{\delta- \alpha}} \, d \Ha_\infty^{\delta} (y)  
\end{equation}
for all $x\in\Rn$.
We estimate this Riesz potential by a Hedberg pointwise inequality  where
the integrals are considered as Choquet integrals.
The classical version of this pointwise estimate goes back to  Hedberg \cite[Theorem~1]{Hed72}. We use the fractional maximal function  by following the idea of  Adams \cite[Proposition 3.1]{Adams1975}.
We refer also to \cite[Lemma 3.6]{HH-S_JFA} and \cite[Lemma 2]{HH-S_AWM}.

\begin{lemma}\label{lem:Choquet-Hedberg}
Let  $n\geq 1$, $\delta \in (0, n]$, $\alpha \in (0, \delta)$, $\kappa \in [0, \alpha)$,  and $p \in [1, \delta/\alpha)$ be given.
Then for all functions $f: \Rn \to [- \infty, \infty]$ and 
for every $x \in \Rn$
\[
\Ri_\alpha^\delta f(x)
\le c 
\bigg(\M^\delta_\kappa f(x)\bigg)^{\frac{\delta-p \alpha}{\delta-\kappa p}}
\bigg(\int_{\Rn} |f(y)|^{p} \, d \Ha^{\delta}_\infty(y) \bigg)^{\frac{\alpha-\kappa}{\delta -\kappa p}},   
\]
where  $c$  is a constant which depends only on $n$, $\alpha$, $\delta$, $\kappa$, and $p$.
\end{lemma}

\begin{proof}
We may assume that $f \in \L^{p}(\Rn, \Ha^{\delta}_\infty)$.
Let $r>0$ and $x\in\Rn$.  Let us  define an annulus  for each $j=1, 2,  \dots$,  such that
$A_j : =\{y \in \Rn : 2^{-j} r \le |x-y|< 2^{-j+1}r\}$. 
Then,  quasi-sublinearity of Choquet integral Theorem~\ref{DennebergThm} yields that
\[
\begin{split}
\int_{B(x, r)} \frac{|f(y)|}{|x-y|^{\delta-\alpha}} \, d \Ha_\infty^{\delta}(y)
&\le c \sum_{j=1}^\infty \int_{A_j} \frac{|f(y)|}{|x-y|^{\delta-\alpha}} \, d\Ha_\infty^{\delta}(y)\\
&\le  \sum_{j=1}^\infty (2^{-j}r)^{\alpha-\delta} \int_{B(x, 2^{-j+1}r)} |f(y)| \, d\Ha_\infty^{\delta}(y)\\
&\le  \frac{2^\delta}{2^\alpha - 2^{\kappa}} r^{\alpha-\kappa} \M_\kappa^\delta f(x), 
\end{split}
\]
where in the last step  the assumption $\kappa<\alpha$  is used in calculation for the sum of a geometric series.

Assume then that $p>1$.
When the integral is taken over the complement of a ball  $B(x,r)$  with respect to $\Rn$, Hölder's inequality (I7) implies that
\[
\begin{split}
&\int_{\Rn\setminus B(x, r)} \frac{|f(y)|}{|x-y|^{\delta-\alpha}} \, d \Ha_\infty^{\delta}(y)\\
&\le 2 \Big(\int_{\Rn\setminus B(x, r)} |f(y)|^{p}\, d\Ha_\infty^{\delta}(y) \Big)^{\frac{1}{p}} 
\Big(\int_{\Rn\setminus B(x, r)} |x-y|^{\frac{p(\alpha-\delta)}{p-1}} \, d \Ha_\infty^{\delta}(y)\Big)^{\frac{p-1}{p}}.
\end{split}
\]
We calculate the last integral on the right-hand side of the above inequality. Let us write that $\eta:= \frac{p(\alpha-\delta)}{p-1}$ and note that $\eta <0$. 
Since $y \not \in B(x,r)$, the supremum  of $|x-y|^\eta$ is $r^\eta$.  Moreover, 
\[
 \big\{ y\in \Rn\setminus B(x, r) : |x-y|^\eta > t \big\}   \subset B\big(x, t^{\frac1\eta}\big) \setminus B(x, r)\,,
\]
where $t>0$.
Hence we obtain
\[
\begin{split}
\int_{\Rn\setminus B(x, r)} |x-y|^{\frac{p(\alpha-\delta)}{p-1}} \, d \Ha_\infty^{\delta}(y)
& = \int_0^\infty \Ha_\infty^{\delta}\Big( \big\{ y\in \Rn\setminus B(x, r) : |x-y|^\eta > t \big\} \Big) \,  dt \\
& = \int_0^{r^\eta} \Ha_\infty^{\delta}\Big( \big\{ y\in \Rn\setminus B(x, r) : |x-y|^\eta > t \big\} \Big) \,   dt \\
& \le \int_0^{r^\eta} t^{\frac{\delta}{\eta}} \, dt \le c(\delta, \alpha, p) r^{\delta + \eta}
= c(\delta, \alpha, p) r^{\delta + \frac{p(\alpha-\delta)}{p-1}}.
\end{split}
\]
The last integral above  converges, since
the assumption  $p<\delta/\alpha$ implies that $\frac{\delta}{\eta}>-1$.
Thus we have the estimate
\begin{equation}\label{equ:Hedberg-estimate}
\Ri_\alpha^\delta f(x) \le c(\delta, \alpha, \kappa, p)  \big( r^{\alpha-\kappa} \M_\kappa^\delta f(x) +  \|f\| r^{\alpha-\frac{\delta}{p}} \big),
\end{equation}
where $\|f\| := \Big(\int_{\Rn} |f(y)|^{p} \, d \Ha^{\delta}_\infty \Big)^{\frac{1}{p}}$.

If $p=1$, then 
\[
\int_{\Rn\setminus B(x, r)} \frac{|f(y)|}{|x-y|^{\delta-\alpha}} \, d \Ha_\infty^{\delta}(y) \le r^{\alpha-\delta}
\int_{\Rn}|f(y)| \, d \Ha_\infty^{\delta}(y).
\]
Hence we obtain  inequality \eqref{equ:Hedberg-estimate} also in this case.

We may assume that $\|f\| >0$.
Choosing 
\[
r= \bigg( \frac{\M_\kappa^\delta f(x)}{\|f\|} \bigg)^{-\frac{p}{\delta - \kappa p}}
\]
in  \eqref{equ:Hedberg-estimate} implies that
\[
\Ri_\alpha^\delta f(x) \le c (\M_\kappa^\delta f(x))^{1-\frac{p(\alpha-\kappa)}{\delta- \kappa p}} \|f\|^{\frac{p(\alpha-\kappa)}{\delta-\kappa p}}
\]
 for all $x \in \Rn$.
 This inequality yields the claim.
\end{proof}

Theorem~\ref{thm:M-bounded}  combined with Lemma \ref{lem:Choquet-Hedberg} gives Theorem
\ref{thm:Choquet-Hedberg}.
We refer to \cite[Remark 3.13]{HH-S_JFA}
for a corresponding result where the classical Riesz potential defined without Choquet integral and  its boundedness with Choquet integrals and with respect to Hausdorff content
is considered. This previous result 
follows from \cite[Lemma 3.6]{HH-S_JFA} and 
\cite[Theorem 7(a)]{Adams1998}. 
The other difference  between
Theorem \ref{thm:Choquet-Hedberg} and \cite[Remark 3.13]{HH-S_JFA}
 is that the functions are not necessarily Lebesgue measurable in the following theorem.

\begin{theorem}\label{thm:Choquet-Hedberg}
Let  $n\geq 1$, $\delta \in (0, n]$, $\alpha \in (0, \delta)$,  and $p \in (\delta/n, \delta/\alpha)$
 be given. Then for all functions $f: \Rn \to [- \infty, \infty]$
\begin{equation}\label{CH}
\bigg(\int_\Rn \big(\Ri_\alpha^n f(x) \big)^{\frac{\delta p}{\delta-p \alpha}}\, d \Ha_\infty^{\delta} \bigg)^{\frac{\delta-p \alpha}{\delta p}}
\le c \bigg(\int_{\Rn} |f(x)|^{p} \, d \Ha^{\delta}_\infty \bigg)^{\frac1p }\,,
\end{equation}
where $c$ is a constant which depends only on $n$, $\delta$,  $\alpha$, $\kappa$,  
and $p$.
\end{theorem}

\begin{proof}
We may assume that $f \in \L^{p}(\Rn, \Ha^{\delta}_\infty)$.
Let  $q \in [1, n/\alpha)$  be fixed  with $\alpha\in (0,n)$. Applying  Lemma~\ref{lem:Choquet-Hedberg} with $\delta =n$,  $\alpha\in (0,n)$, and
$\kappa\in [0,n)$
for this parameter $q$
implies that
 for all $x\in\Rn$
\begin{equation}\label{Case_n_meas-2}
\Ri_\alpha^n f(x)
\le c 
\bigg(\M^n_\kappa f(x)\bigg)^{\frac{n-q \alpha}{n-\kappa q}}
\bigg(\int_{\Rn} |f(y)|^{q} \, d \Ha^{n}_\infty(y) \bigg)^{\frac{\alpha-\kappa}{n -\kappa q}}\,,
\end{equation}
whenever $f \in \L^q(\Rn, \Ha^{n}_\infty)$.
Now we  replace $q$
in \eqref{Case_n_meas-2}
by $pn/\delta$. Since 
$q \in [1, n/\alpha )$ in \eqref{Case_n_meas-2},
$1 \le pn/\delta < n/\alpha$ i.e.\ $\delta/n \le p < \delta/\alpha$.
The assumptions of the present theorem means that here
$\delta\in (0,n]$, $\alpha\in (0,\delta )$, and $\kappa\in [0,\alpha )$.    
Hence, for all $x\in\Rn$
\[
\big(\Ri_\alpha^n f(x) \big)^{\frac{\delta - \kappa p}{\delta-p \alpha}}
\le c 
\Big(\M^n_\kappa f(x)\Big)
\bigg(\int_{\Rn} |f(y)|^{\frac{pn}\delta} \, d \Ha^{n}_\infty(y) \bigg)^{\frac{\alpha - \kappa }{n - n p \alpha/\delta}}\,,
\]
whenever $f \in \L^{pn/\delta}(\Rn, \Ha^{n}_\infty)$.
 As before we note that
 by
 Proposition~\ref{GeneralizationOV}
the inclusion  $\L^{p}(\Rn, \Ha^{\delta}_\infty) \subset \L^{pn/\delta}(\Rn, \Ha^{n}_\infty)$ holds.
Estimating the last integral by Proposition~\ref{GeneralizationOV} and  then taking both sides to the $p$th power imply that
 for all $x\in\Rn$
\[
\big(\Ri_\alpha^n f(x) \big)^{\frac{p (\delta- \kappa p) }{\delta-p \alpha}}
\le c(n, \delta)
\Big(\M^n_\kappa f(x)\Big)^p
\bigg( \int_{\Rn} |f(y)|^{p} \, d \Ha^{\delta}_\infty (y)\bigg)^{\frac{p (\alpha- \kappa)}{\delta -  p \alpha}} \,,
\]
whenever $f \in \L^p(\Rn, \Ha^{\delta}_\infty)$.
Then we choose $\kappa =0$ and 
use the boundedness of the Hausdorff content maximal function Theorem~\ref{thm:M-bounded} to obtain
\[
\begin{split}
\int_\Rn \big(\Ri_\alpha^n f(x) \big)^{\frac{p\delta}{\delta-p \alpha}}\, d \Ha_\infty^{\delta}
&\le c \bigg(\int_{\Rn} |f(y)|^{p} \, d \Ha^{\delta}_\infty(y) \bigg)^{\frac{p\alpha}{\delta - p \alpha}}
\int_\Rn \Big(\M^n f(x)\Big)^{p}\, d \Ha_\infty^{\delta}(x)\\
&\le c \bigg(\int_{\Rn} |f(y)|^{p} \, d \Ha^{\delta}_\infty(y) \bigg)^{\frac{p\alpha}{\delta - p \alpha}}
\int_\Rn  |f(x)|^{p}\, d \Ha_\infty^{\delta}(x)\\
&\le c \bigg(\int_{\Rn} |f(y)|^{p} \, d \Ha^{\delta}_\infty (y)\bigg)^{\frac1p \frac{p \delta}{\delta - p \alpha}} \,,
\end{split}
\]
whenever $f \in \L^p(\Rn, \Ha^{\delta}_\infty)$.
\end{proof}

\begin{remark}\label{SmallComments}
(1) By the proof  of the previous theorem 
inequality  \eqref{CH} is possibly  valid for a larger class of functions than stated,  namely for functions in 
$\L^{pn/\delta}(\Rn, \Ha^{n}_\infty)$,  because of Proposition \ref{GeneralizationOV}.
\\
(2) 
We point out that the
exponent $\frac{\delta p}{\delta-p \alpha}$  in the left-hand side of \eqref{CH} in Theorem~\ref{thm:Choquet-Hedberg} is larger than the exponent $\frac{\delta p}{ \delta - p \kappa }$ 
in the left-hand side of \eqref{strongineq1b}
in Theorem~\ref{thm:fractional-M-bounded} where maximal functions were considered.
\end{remark}

 Next two theorems consider  Lebesgue measurable functions.
 By \eqref{compareMM^n} we have $\M^n_\kappa f(x) \approx  M_\kappa f(x)$, 
where $\kappa \in [0, n)$.
That is,  for Lebesgue measurable functions these two maximal operators are comparable to each other with constants depending only on $n$.
Also, the  Hausdorff content Riesz potential  $\Ri_\alpha^{\delta} f$ whenever $\delta =n$ 
 is comparable to the classical 
$\alpha$-dimensional Riesz potential
\[
x \mapsto \int_\Rn \frac{|f(y)|}{|x-y|^{n- \alpha}} \, dy
\]
for all Lebesgue measurable $f: \Rn \to [-\infty, \infty]$.

Adams studied boundedness of $M_\kappa$ in \cite{Adams1998}. By combining his boundedness theorem
\cite[Theorem 7]{Adams1998} with   Lemma~\ref{lem:Choquet-Hedberg}
we obtain the following result for measurable functions.
The case $\alpha =1$ was given already  in  \cite[Remark 3.13]{HH-S_JFA}.

\begin{remark}
The difference between Theorem~\ref{thm:Choquet-Hedberg} and Theorem~\ref{thm:Choquet-Hedberg-measurable}   is that the latter covers only Lebesgue measurable functions
and in 
Theorem~\ref{thm:Choquet-Hedberg-measurable}  a parameter $\kappa>0$ is involved. This means that 
the dimension of the Hausdorff content 
could be taken smaller on the left-hand side of 
\eqref{rmk313}  than on  its right-hand side where as 
the dimension of the Hausdorff content stays the same on the both sides of  inequality
\eqref{CH}.
\end{remark}

\begin{theorem}\label{thm:Choquet-Hedberg-measurable}
If  $n\geq 1$, $\delta \in (0, n]$, $\alpha \in (0, \delta)$, $\kappa \in [0, \alpha)$,  and $p \in (\delta/n, \delta/\alpha)$
 are given, then 
for  all Lebesgue measurable  functions $f: \Rn \to [- \infty, \infty]$
\begin{equation}\label{rmk313}
\bigg(\int_\Rn \big(\Ri_\alpha^n f(x) \big)^{\frac{p(\delta- \kappa p)}{\delta-p \alpha}}\, d \Ha_\infty^{\delta- \kappa p} \bigg)^{\frac{\delta - p \alpha}{p (\delta - \kappa p) }}
\le c \bigg(\int_{\Rn} |f(x)|^{p} \, d \Ha^{\delta}_\infty \bigg)^{\frac1p }\,,
\end{equation}
where $c$ is a constant which depends only on $n$, $\delta$, $\alpha$,  $\kappa$, and $p$.
\end{theorem}

\begin{proof}
The proof is similar to the proof of Theorem~\ref{thm:Choquet-Hedberg}.
 Now the equivalence  $\M^n_\kappa f(x) \approx M_\kappa f(x)$  is used and  
\cite[Theorem 7(a)]{Adams1998} is applied
 instead of Theorem~\ref{thm:M-bounded}.
\end{proof}

We point out that
inequality  \eqref{rmk313}  is possibly  valid for a larger class of functions than stated,  namely for functions  $f$ satisfying
the integrability condition 
$ \int_{\Rn} |f(x)|^{pn/\delta} \, d \Ha^{n}_\infty < \infty$.

Next we show a boundedness result to the Hausdorff content  Riesz potential $\Ri_\alpha^{\delta}f$
for any $\delta \in (0,n)$ which can be seen 
as a generalisation of  \cite[Theorem 2.8.4]{Ziemer89}.
The boundedness inequality follows from \cite[Theorem 1.2]{ChenOoiSpector2023}.

\begin{theorem}\label{thm:GeneralRiesz_bdd}
If  $n\geq 1$, $\delta \in (0, n)$, $\alpha \in (0, \delta)$,  and $p \in (1, \delta/\alpha)$
 are given, then 
for  all Lebesgue measurable  functions   $f: \Rn \to [- \infty, \infty]$
\begin{equation*}
\bigg(\int_\Rn \big(\Ri_\alpha^{\delta}f(x) \big)^{\frac{\delta p}{\delta-p \alpha}}\, d \Ha_\infty^{\delta} \bigg)^{\frac{\delta - p \alpha}{\delta p}}
\le c \bigg(\int_{\Rn} |f(x)|^{p} \, d \Ha^{\delta}_\infty \bigg)^{\frac1p }\,,
\end{equation*}
where $c$ is a constant which depends only on $n$, $\delta$, $\alpha$, and $p$.
\end{theorem}

\begin{proof}
We may assume that $\int_{\Rn} |f(x)|^{p} \, d \Ha^{\delta}_\infty < \infty$.
By Lemma~\ref{lem:Choquet-Hedberg} with $\kappa\in [0,\alpha )$
 for all $x\in\Rn$
\begin{equation*}
\Ri_\alpha^{\delta} f(x)
\le c 
\bigg(\M^{\delta}_\kappa f(x)\bigg)^{\frac{\delta-p\alpha}{\delta-\kappa p}}
\bigg(\int_{\Rn} |f(y)|^{p} \, d \Ha^{\delta}_\infty (y)\bigg)^{\frac{\alpha-\kappa}{\delta -\kappa p}}.
\end{equation*}
Taking both sides to the  power $p(\delta-\kappa p)/(\delta-p\alpha)$ imply that
 for all $x\in\Rn$
\[
\big(\Ri_\alpha^{\delta} f(x) \big)^{\frac{p (\delta- \kappa p) }{\delta-p \alpha}}
\le c(n, \alpha, \delta)
\Big(\M^{\delta}_\kappa f(x)\Big)^p
\bigg( \int_{\Rn} |f(y)|^{p} \, d \Ha^{\delta}_\infty (y)\bigg)^{\frac{p (\alpha- \kappa)}{\delta -  p \alpha}} \,.
\]
By using the boundedness of the maximal function \cite[Theorem 1.2]{ChenOoiSpector2023} with $\kappa =0$
we obtain
\[
\begin{split}
\int_\Rn \big(\Ri_\alpha^{\delta} f(x) \big)^{\frac{p\delta}{\delta-p \alpha}}\, d \Ha_\infty^{\delta}
&\le c \bigg(\int_{\Rn} |f(y)|^{p} \, d \Ha^{\delta}_\infty \bigg)^{\frac{p \alpha}{\delta - p \alpha}}
\int_\Rn \Big(\M^{\delta}f(x)\Big)^{p}\, d \Ha_\infty^{\delta}\\
&\le c \bigg(\int_{\Rn} |f(y)|^{p} \, d \Ha^{\delta}_\infty \bigg)^{\frac1p \frac{p \delta }{\delta - p \alpha}}. 
\end{split}
\]
This gives the result.
\end{proof}

\begin{remark}
(1) Inequality \eqref{rmk313}  with $\delta =n$ and $\alpha =1$ appears in \cite[Remark 3.13]{HH-S_JFA} where the classical $1$-dimensional Riesz potential is used.\\
(2) Theorem~\ref{thm:Choquet-Hedberg-measurable} recovers the classical result for the boundedness of $\alpha$-dimensional Riesz potential
\cite[Theorem 2.8.4]{Ziemer89}, $\alpha \in (0, n)$.
Namely, choosing $\delta=n$ and $\kappa =0$ in inequality \eqref{rmk313} implies
\[
\bigg(\int_\Rn \bigg(\int_\Rn \frac{|f(y)|}{|x-y|^{n- \alpha}} \, dy \bigg)^{\frac{np}{n-p \alpha}}\, dx \bigg)^{\frac{n - p \alpha}{n p}}
\le c \bigg(\int_{\Rn} |f(x)|^{p} \, d x\bigg)^{\frac1p }
\]
for all  Lebesgue measurable functions $f: \Rn \to [- \infty, \infty]$. \\
(3)  
 We show that 
 \cite[Corollary 1.4]{MartinezSpector2021} implies a limit case for Theorem
\ref{thm:Choquet-Hedberg-measurable}.
Let $\delta =n$ in the definition of $\Ri^\delta_\alpha$ where $\alpha \in (0,n)$.
Let $\Omega $ be an open set in $\Rn$ and let us study  compactly supported  Lebesgue measurable functions
$f: \Omega \to [- \infty, \infty]$ such that
$\int_\Omega |f(x)|^{\frac \delta\alpha} \,d \Ha^\delta_\infty \le 1$.
By Proposition~\ref{GeneralizationOV}  also $\int_\Omega |f(x)|^{\frac n\alpha} \, dx \le c(n, \delta)$.
Let $\beta \in (0,n]$ be given.
Applying
\cite[Corollary 1.4]{MartinezSpector2021}  for $p=\delta /\alpha >1$  
yields the following result.
There exist constants $c_1, c_2>0$  that depend on $\alpha$, $n$, $\beta$, and $\Omega$ such that
\begin{equation}\label{limitcase}
\int_\Omega \exp \Big(c_1 \big(\Ri_\alpha^n f(x)\big)^{\frac{n}{n-\alpha}} \Big) \, d \Ha^\beta_\infty < c_2
\end{equation}
for all  compactly supported Lebesgue measurable  functions $f: \Omega \to [- \infty, \infty]$ with
 the integrability condition
$\int_\Omega |f(x)|^{\frac \delta\alpha} \, d\Ha^\delta_\infty\le 1.$
Inequality \eqref{limitcase} can be seen as a limit case to Theorem~\ref{thm:Choquet-Hedberg-measurable} when 
$p= \delta/\alpha$.
\end{remark}

\bibliographystyle{amsalpha}

\end{document}